\documentclass[12pt, reqno, twoside, letterpaper]{amsart}

\usepackage[text={144mm,233mm},centering]{geometry} 

\usepackage{setspace}
\usepackage{amsthm}
\newtheoremstyle{mystyle}
{1em}                      
{1em}                      
{\itshape}                 
{}                         
{\scshape}                 
{.}                        
{5pt plus 1pt minus 1pt}   
{}                         

\newtheoremstyle{mystyle2}
{1em}                    
{1em}                    
{}                       
{}                       
{\scshape}               
{.}                      
{5pt plus 1pt minus 1pt} 
{}                       

\theoremstyle{mystyle}

\newtheorem{X}{X}                     
\newtheorem{theorem}[X]{Theorem}

\newtheorem*{theorem-cited*}{Theorem} 

\newtheorem{Y}{Y}
\newtheorem{corollary}[Y]{Corollary}

\newtheoremstyle{acknowledgements}  
{1em}                               
{1em}                               
{}                                  
{}                                  
{\bfseries}                         
{.}                                 
{5pt plus 1pt minus 1pt}            
{}                                  

\theoremstyle{acknowledgements}

\newtheorem*{acknowledgements*}{Acknowledgements} 

\usepackage{color}
\usepackage{hyperref}
\hypersetup{
    pdftoolbar=true,                        
    pdfmenubar=false,                        
    pdffitwindow=false,                      
    pdfstartview={FitH},                    
    pdftitle={Consecutive primes in tuples},    
    pdfauthor={William D. Banks, Tristan Freiberg & Caroline L. Turnage--Butterbaugh},    
    pdfcreator={William D. Banks, Tristan Freiberg & Caroline L. Turnage--Butterbaugh},
    pdfproducer={William D. Banks, Tristan Freiberg & Caroline L. Turnage--Butterbaugh},
    pdfnewwindow=true,                       
    colorlinks=true,
    linkcolor={black},                      
    citecolor={black},
    filecolor={black},
    urlcolor={black},         
}

\usepackage{amsmath}
\usepackage{amssymb}                      
\usepackage{mathabx}                      
\usepackage{mathrsfs}                     
\usepackage{cite}                         
\usepackage{enumitem}                     
\usepackage{caption}                      

\renewcommand{\le}{\ensuremath{\leqslant}} 
\renewcommand{\ge}{\ensuremath{\geqslant}} 
\newcommand{\NN}{\ensuremath{\mathbb{N}}}  
\newcommand{\ZZ}{\ensuremath{\mathbb{Z}}}  
\newcommand{\e}{\ensuremath{\mathrm{e}}}   
\newcommand{\cH}{\ensuremath{\mathcal{H}}} 
\newcommand{\cA}{\ensuremath{\mathcal{A}}} 
\newcommand{\cB}{\ensuremath{\mathcal{B}}} 
\newcommand{\cS}{\ensuremath{\mathcal{S}}} 

\title[Consecutive primes in tuples]
      {Consecutive primes in tuples}

\author[W. D. Banks]{William D. Banks}
\address{Department of Mathematics, 
         University of Missouri, Columbia MO, USA.}
\email{bankswd@missouri.edu}

\author[T. Freiberg]{Tristan Freiberg}
\address{Department of Mathematics, 
         University of Missouri, Columbia MO, USA.}
\email{freibergt@missouri.edu}

\author[C. L. Turnage-Butterbaugh]{Caroline L. Turnage--Butterbaugh}
\address{Department of Mathematics, 
         University of Mississippi, University MS, USA.}
\email{clbutter@olemiss.edu}

\thanks{CLT-B is supported by a GAANN fellowship (grant no.\ 
        P200A90092).}

\date{\today}

\begin{document}

\begin{abstract}
In a  stunning new advance towards the Prime $k$-tuple Conjecture,
Maynard  and Tao  have shown that if $k$  is sufficiently large in 
terms of $m$, then for an admissible $k$-tuple 
$
\cH(x) = \{gx + h_j\}_{j=1}^k
$
of linear forms  in $\ZZ[x]$, the  set  
$
\cH(n) = \{gn + h_j\}_{j=1}^k
$ 
contains at least  $m$  primes for infinitely many  $n  \in  \NN$.
In this  note, we deduce that 
$
\cH(n) = \{gn + h_j\}_{j=1}^k
$ 
contains at least $m$ {\em consecutive} primes for infinitely many 
$n \in \NN$.
We answer  an  old question  of Erd\H os  and Tur\'an by producing  
strings  of  $m + 1$  consecutive  primes  whose  successive  gaps 
$
\delta_1,\ldots,\delta_m
$ 
form an increasing (resp.\ decreasing) sequence.  
We also show that such strings exist with 
$
\delta_{j-1} \mid \delta_j
$ 
for $2 \le j \le m$.  
For any  coprime  integers $a$  and $D$ we  find  arbitrarily long 
strings of consecutive primes  with bounded gaps in the congruence 
class $a \bmod D$. 
\end{abstract}

\maketitle

\section{Introduction and statement of results}

We say that a  $k$-tuple of  linear forms in  $\ZZ[x]$, denoted by
\[
\cH(x) = \{g_jx + h_j\}_{j=1}^k, 
\]
is {\em admissible} if the associated polynomial
$
f_\cH(x) = \prod_{1 \le j \le k}(g_jx + h_j)
$ 
has no fixed prime divisor, that is, if the inequality
\[
\#\{n \bmod p : f_\cH(n) \equiv 0 \bmod p\} < p
\]
holds for every prime number $p$.  
In this note we consider only $k$-tuples for which
\begin{equation}
 \label{eq:1}
 g_1,\ldots,g_k > 0 
  \qquad\text{and}\qquad
   {\textstyle \prod_{1 \le i < j\le k} } (g_ih_j - g_jh_i) \ne 0.
\end{equation}
One form of the  {\em Prime $k$-tuple Conjecture}  asserts that if 
$\cH(x)$ is admissible and satisfies \eqref{eq:1}, then 
$
\cH(n) = \{g_jn + h_j\}_{j=1}^k
$
is a $k$-tuple of primes for infinitely many $n \in \NN$.
Recently,   Maynard \cite{5}  and  Tao  have made   great  strides 
towards proving this form of the Prime $k$-tuple Conjecture, which 
rests among the greatest unsolved problems in number theory.
The  following  formulation of their  remarkable  theorem has been 
given by Granville \cite[Theorem 6.2]{3}.

\begin{theorem-cited*}[Maynard--Tao]
For  any  $m \in \NN$  with  $m \ge 2$  there is a  number  $k_m$,
depending only on $m$, such that the following holds for every 
integer $k \ge k_m$.  
If   $\{g_jx + h_j\}_{j=1}^k$   is   admissible    and   satisfies
\eqref{eq:1},  then  $\{g_jn + h_j\}_{j=1}^k$  contains $m$ primes
for infinitely many $n \in \NN$.
In  fact,  one  can  take   $k_m$  to  be  any  number  such  that
$k_m\log k_m > \e^{8m + 4}$.
\end{theorem-cited*}

Zhang   \cite[Theorem 1]{10}   was   the   first   to  prove  that 
$
\liminf_{n \to \infty}(p_{n+1} -  p_n)
$ 
is   bounded;  he   showed  that   for  an   admissible  $k$-tuple 
$\cH(x)  =  \{x  +  b_j\}_{j=1}^k$  there  exist  infinitely  many
integers  $n$ such that  $\cH(n)$ contains  at  least two  primes,
provided that $k \ge 3.5 \times 10^6$.
Zhang's  proof  was  subsequently  refined  in  a polymath project
\cite[Theorem 2.3]{7}   to  the  point   where  one   could   take 
$k_2 = 632$ (at  least in the case of monic linear forms).
Maynard  \cite[Propositions 4.2, 4.3]{5}  has  shown  that one can 
take  $k_2 = 105$  and  $k_m = cm^2\e^{4m}$  in  the  Maynard--Tao 
theorem, where $c$ is an absolute (and effective) constant.
Another  polymath  project \cite[Theorem 3.2]{8} has since refined  
Maynard's   work  so   that   one   can   take  $k_2  =  50$   and 
$k_m = c\e^{(4 - 28/157)m}$.
(In  \cite{5,8},  only  tuples of  monic  linear forms are treated  
explicitly,  although  the results should extend to general linear 
forms as considered in \cite{3}.)

The  purpose of the  present  note is to explain some  interesting
consequences of the Maynard--Tao theorem.
We  refer  the  reader to the  expository  article  \cite{3} of 
Granville  for the recent  history and  ideas  leading up  to this 
breakthrough  result, as well  as a  discussion of  its  potential 
impact. 
Without  doubt,  this  result  and  its  proof  will have numerous 
applications,   many   of  which   have  already  been   given  in 
\cite{3}.  
We  are   grateful  to  Granville   for  pointing  out  to us that 
Corollary \ref{cor:1} (below) can now be proved. 

The following theorem establishes the existence of $m$-tuples that 
infinitely  often  represent strings of  {\em  consecutive}  prime 
numbers.

\begin{theorem}
 \label{thm:1}
Let $m, k \in \NN$ with $m \ge 2$  and $k \ge k_m$, where $k_m$ is 
as in the Maynard--Tao theorem.
Let   $b_1, \ldots, b_k$    be   distinct   integers   such   that 
$\{x + b_j\}_{j=1}^k$  is admissible, and let  $g$ be any positive 
integer coprime with $b_1\cdots b_k$.
Then, for some subset 
$
\{h_1,\ldots,h_m\} \subseteq \{b_1,\ldots,b_k\},
$ 
there    are    infinitely    many    $n  \in  \NN$    such   that 
$gn + h_1,\ldots,gn + h_m$ are consecutive primes.
\end{theorem}

A special case of Theorem~\ref{thm:1}, with $m = 2$,  $g = 1$ (and 
the  weaker  bound $k_2 \ge 3.5 \times 10^6$),  has  already  been 
established  in recent  work of  Pintz \cite[Main Theorem]{6}, 
which is based on  Zhang's method but uses a different argument to 
the one presented here. 

Theorem \ref{thm:1} (which is proved in \S\ref{sec:2}) has various 
applications to the study of gaps between consecutive primes.
To    state   our    results,   let    us    call    a    sequence 
$(\delta_j)_{j=1}^{m}$   of  positive  integers  a  { \em  run  of 
consecutive prime gaps} if
\[
 \delta_{j} = d_{r + j} = p_{r+j+1} - p_{r+j} 
  \qquad (1 \le j \le m)
\]
for some  natural  number  $r$,  where  $p_n$  denotes  the $n$-th 
smallest prime.
The  following  corollary of  Theorem  \ref{thm:1}  answers an old 
question of Erd\H os  and  Tur\'an\cite{1} (see also Erd\H os 
\cite{2}  and  Guy \cite[A11]{4}).

\begin{corollary}
 \label{cor:1} 
For    every   $m \ge 2$   there   are   infinitely    many   runs 
$(\delta_j)_{j=1}^{m}$   of    consecutive    prime    gaps   with
$\delta_1 < \cdots < \delta_m$,  and  infinitely  many  runs  with
$\delta_1 > \cdots > \delta_m$.
\end{corollary}

Moreover, in the proof (see \S\ref{sec:2}) we construct infinitely 
many runs $(\delta_j)_{j=1}^{m}$  of consecutive  prime  gaps with
\[
 \delta_1 + \cdots + \delta_{j-1} < \delta_j 
  \qquad (2 \le j \le m),
\] 
and infinitely many runs with
\[
 \delta_j > \delta_{j+1} + \cdots + \delta_m
  \qquad (1\le j \le m-1).
\]

Using a similar argument, we can impose a divisibility requirement 
amongst gaps between consecutive primes as well.

\begin{corollary}
 \label{cor:2} 
For   every   $m  \ge  2$   there   are   infinitely   many   runs 
$(\delta_j)_{j=1}^{m}$   of  consecutive   prime  gaps  such  that 
$\delta_{j-1}\mid \delta_j$ for $2 \le j \le m$,   and  infinitely 
many   runs   such   that   $\delta_{j+1}   \mid   \delta_j$   for 
$1 \le j \le m - 1$.
\end{corollary}

In the proof (see \S\ref{sec:2}) we construct infinitely many runs 
$(\delta_j)_{j=1}^{m}$   of   consecutive    prime    gaps    with
$\delta_1\cdots \delta_{j-1} \mid \delta_j$ for $2 \le j \le m$, 
and infinitely many runs with 
$
\delta_{m}\delta_{m-1} \cdots \delta_{j+1} \mid \delta_j
$  
for $1 \le j \le m-1$.

As  another  application  of Theorem \ref{thm:1}, in \S\ref{sec:2}
we prove the following  extension of a result of  Shiu \cite{9} on 
consecutive primes in a given congruence class. 

\begin{corollary}
 \label{cor:3}
Let $a$ and $D\ge 3$ be coprime integers.  
For  every $m \ge 2$, there are  infinitely  many $r \in \NN$ such 
that 
$
 p_{r+1} \equiv p_{r+2} 
   \equiv \cdots \equiv p_{r+m} \equiv a \bmod D
$
and  $p_{r+m}  -  p_{r+1}  \le  DC_m$,  where  $C_m$ is a constant 
depending only on $m$.
\end{corollary}
\noindent 
Shiu \cite{9}  attributes to  Chowla the conjecture that there are  
infinitely many pairs  of consecutive primes $p_r$ and $p_{r + 1}$  
with   $p_r  \equiv  p_{r + 1}  \equiv  a  \bmod D$    (see   also 
\cite[A4]{4}), and proved the  above result without the constraint 
$p_{r + m} - p_{r + 1} \le DC_m$.

\section{Proofs}
 \label{sec:2}

\begin{proof}[Proof of Theorem \ref{thm:1}]
Replacing each  $b_j$  with $b_j + gN$ for a suitable integer $N$, 
we    can    assume    without    loss    of    generality    that
\[
 1 < b_1 < \cdots < b_k.
\]
Let $\cS$ be the set of integers  $t$ such that $1 \le t \le b_k$,
$t \not\in \{b_1,\ldots,b_k\}$. 
Let $\{q_t : t \in \cS\}$  be distinct primes coprime to  $g$ such
that   $t  \not\equiv  b_j  \bmod  q_t$   for   all   $t \in \cS$, 
$1\le j \le k$. 
By the Chinese remainder theorem we can find an integer  $a$  such 
that 
\begin{equation}
 \label{eq:2}
  ga + t \equiv 0 \bmod q_t \qquad (t \in \cS),
\end{equation}
and therefore
\begin{equation}
 \label{eq:3}
  ga + b_j \not\equiv 0 \bmod q_t 
   \qquad (t \in \cS, 1\le j \le k).
\end{equation}

Consider the $k$-tuple 
\[
 \cA(x) = \{gQx + ga + b_j\}_{j=1}^{k}
  \qquad\text{where}\qquad 
   Q = {\textstyle \prod_{t \in \cS} }\, q_t.
\]
In     view     of     \eqref{eq:3}     and    the    fact    that 
$\gcd(g,b_1\cdots b_k) = 1$,  we  have $\gcd(gQ,ga + b_j) = 1$ for 
each  $j$,  and  since  $\{x + b_j\}_{j=1}^k$  is  admissible,  it 
follows that the $k$-tuple $\cA(x)$ is also admissible. 
Moreover,  $\cA(x)$  satisfies  \eqref{eq:1}  (with $g_j = gQ$ and 
$h_j = ga + b_j$) as the integers  $b_1,\ldots,b_k$  are  distinct 
and $gQ \ge 1$.

For  every  $N \in \NN$,  the  congruences  \eqref{eq:2}  and  our 
choices     of       $Q$       and       $a$       imply      that
\[
 g(QN + a) + t \equiv 0 \bmod q_t \qquad (t \in \cS).
\]
Consequently,    any    prime     number    in    the     interval 
$[g(QN + a) + b_1, g(QN + a) + b_k]$ must lie in $\cA(n)$.
Let $m'$  be the largest integer  for which there  exists a subset 
$
\{h_1,\ldots,h_{m'}\} \subseteq \{b_1,\ldots,b_k\}
$ 
with      the        property        that        the       numbers
\begin{equation}
 \label{eq:4}
 g(QN + a) + h_i \qquad (1 \le i \le m')
\end{equation}
are simultaneously prime for infinitely many $N \in \NN$.
Since  $k \ge k_m$  we can  apply the  Maynard--Tao  theorem  with 
$\cA(x)$ to deduce that $m' \ge m$.

By the maximal property of  $m'$, it must be the case that for all 
sufficiently large $N \in \NN$, if the numbers in \eqref{eq:4} are 
all   prime,  then  $g(QN + a) + b_j$  is   composite   for  every
$
 b_j \in \{b_1,\ldots,b_k\} \setminus \{h_1,\ldots,h_{m'}\}.
$
Hence,  for   infinitely    many  $N  \in  \NN$,   the    interval
$[g(QN + a) + b_1, g(QN + a) + b_k]$   contains   precisely   $m'$ 
primes,   namely,   the   numbers  $\{gn + h_i\}_{i=1}^{m'}$  with 
$n = QN + a$.
\end{proof}

\begin{proof}[Proof of Corollary \ref{cor:1}]
Let $m \ge 2$, and let $k \ge k_{m+1}$.
Let $\cA(x) = \{x + 2^j\}_{j=1}^{k}$,  which  is easily seen to be 
admissible.
By   Theorem   \ref{thm:1},    there   exists   an   $(m+1)$-tuple
\[
 \cB(x) = \{x + 2^{\nu_j}\}_{j=1}^{m+1} \subseteq \cA(x)
\]
such that  $\cB(n)$  is an $(m+1)$-tuple of consecutive primes for 
infinitely many $n$. 
Here, $1 \le \nu_1 < \cdots < \nu_{m + 1} \le k$.
For such $n$, writing 
\[
 \cB(n) 
 = \{n + 2^{\nu_j}\}_{j=1}^{m + 1}
 = \{p_{r + 1},\ldots,p_{r + m + 1}\}
\]
with some integer $r$, we have
\[
 \delta_j 
  = d_{r+j} 
   = p_{r + j +1} - p_{r + j} 
    = 2^{\nu_{j + 1}} - 2^{\nu_j}
  \qquad(1\le j\le m).
\]
Then
\[
  \sum_{i=1}^{j-1} \delta_i
 = \sum_{i=1}^{j-1} (2^{\nu_{i + 1}} - 2^{\nu_i})
 = 2^{\nu_j} - 2^{\nu_1}
 < 2^{\nu_{j + 1}} - 2^{\nu_j}
 = \delta_j
    \qquad (2 \le j \le m).
\]
Hence, 
$
\delta_{j-1} \le \delta_1 + \cdots + \delta_{j-1} < \delta_j
$ 
for each $j$, which proves the first statement.
To    obtain    runs    of    consecutive     prime   gaps    with 
$
\delta_j > \delta_{j + 1} + \cdots + \delta_m \ge \delta_{j + 1},
$ 
consider instead the admissible $k$-tuple $\{x - 2^j\}_{j=1}^{k}$. 
This completes the proof.
\end{proof}

\begin{proof}[Proof of Corollary \ref{cor:2}]
Let $m \ge 2$, and let $k \ge k_{m + 1}$.
Put   $Q   =   \prod_{p \le k}  p$,  and   define   the   sequence 
$b_1,\ldots,b_k$ inductively as follows. 
Let
\[
 b_1 = 0, \quad  b_2 = Q, \quad b_3 = 2Q,
\]
and for any $j \ge 3$ let 
\[
 b_j = b_{j - 1} + \prod_{1 \le s < t \le j - 1} (b_t - b_s).
\]
Note that
\begin{equation}
 \label{eq:5}
  (b_{u + 1} - b_u) \mid (b_{v + 1} - b_v) \qquad (v \ge u \ge 1).
\end{equation}

Now put $\cA(x) = \{x + b_j\}_{j=1}^k$, and  observe that $\cA(x)$
is admissible since $Q$ divides each integer $b_j$. 
By  Theorem   \ref{thm:1},   there   exists   an   $(m + 1)$-tuple 
\[
 \cB(x) = \{ x + b_{\nu_j} \}_{j = 1}^{m + 1} \subseteq \cA(x)
\]
such that $\cB(n)$ is an $(m + 1)$-tuple of consecutive primes for 
infinitely many $n$. 
Here, $1 \le \nu_1 < \cdots < \nu_{m + 1} \le k$.
For any such $n$, writing 
\[
 \cB(n) 
 = \{ n + b_{\nu_j} \}_{j = 1}^{m + 1}
 = \{ p_{r + 1},\ldots,p_{r + m + 1} \}
\]
with some integer $r$, we have
\[
 \delta_j 
  = d_{r + j}
  = p_{r + j + 1} - p_{r + j}
  = b_{\nu_{j + 1}} - b_{\nu_j}
  \qquad (1 \le j \le m).
\]
Then
\[
  \prod_{i=1}^{j-1} \delta_i
 = \prod_{i=1}^{j-1} (b_{\nu_{i + 1}} - b_{\nu_i})
    \,\bigg|\,
     \prod_{1 \le s < t \le \nu_j}(b_t - b_s)
 = b_{\nu_j + 1} - b_{\nu_j}
\]
if $2\le j\le m$. 
On   the    other   hand,   using   \eqref{eq:5}   we   see   that
\[
 (b_{\nu_j + 1} - b_{\nu_j})
  \,\bigg|\,
   \sum_{i = \nu_j}^{\nu_{j + 1} - 1} (b_{i + 1} - b_i)
 = b_{\nu_{j + 1}} - b_{\nu_j}
 = \delta_j.
\]
Hence, 
$
\delta_1 \cdots \delta_{j-1} \mid \delta_j
$ 
for $2 \le j \le m$, which proves the first statement.
To    obtain    runs    of    consecutive    prime    gaps    with 
$
\delta_{m} \delta_{m-1} \cdots \delta_{j + 1} \mid \delta_j
$ 
for $1 \le j \le m - 1$, consider instead the admissible $k$-tuple 
$\{x - b_j\}_{j=1}^{k}$.
The corollary is proved.
\end{proof}

\begin{proof}[Proof of Corollary \ref{cor:3}]
Let $m \ge 2$, and let $k \ge k_m$.
Let $\{x  +  a_j\}_{j = 1}^k$  be  any  admissible  $k$-tuple with 
$a_1 < \cdots < a_k$,  and put $b_j = Da_j+a$ for $1 \le j \le k$; 
then $\{x + b_j\}_{j=1}^k$ is also admissible.
Since  $\gcd(D,b_j) = \gcd(D,a) = 1$  for  each  $j$, we can apply 
Theorem  \ref{thm:1}  with  $g = D$  to  conclude  that there is a 
subset 
$
\{h_1,\ldots,h_m\} \subseteq \{b_1,\ldots,b_k\}
$ 
such that $Dn + h_1, \ldots, Dn + h_m$  are consecutive primes for 
infinitely many $n \in \NN$;  as such primes lie in the arithmetic 
progression $a \bmod D$ and are contained in an interval of length 
$b_k - b_1 = D(a_k - a_1)$, the corollary follows.
\end{proof}

\begin{acknowledgements*}
In  the  first  draft  of  this   manuscript,  we  proved  Theorem 
\ref{thm:1} under the assumption that $k \ge \exp(\e^{12m})$.  
We thank Andrew Granville for showing that $k$  need not be larger 
than  the  number  $k_m$  in  the  Maynard--Tao  theorem  and  for 
simplifying our original proof of Theorem \ref{thm:1}. 
We  also thank  Gergely Harcos, James Maynard, and the referee for 
providing helpful comments on our earlier drafts.
\end{acknowledgements*}

\end{document}